\documentclass[12pt]{amsart}

\usepackage{latexsym}
\usepackage{psfrag}
\usepackage{amsmath}
\usepackage{amssymb}
\usepackage{epsfig}
\usepackage{amsfonts}
\usepackage{amscd}
\usepackage{mathrsfs}
\usepackage{graphicx}
\usepackage{enumerate}

\newtheorem{theorem}{Theorem}%[section]

\newtheorem{cor}[theorem]{Corollary}

\begin{document}%%%%%%%%%%%%%%%%%%%%%%%%%%%%%%%%%%%%%%%%%%%%%%%%%%%%%%%%
%%%%%%%%%%%%%%%%%%%%%%%%%%%%%%%%%%%%%%%%%%%%%%%%%%%%%%%%%%%%%%%%%%%%%%%%

\title[Andrews-Gordon Type Series]{Andrews-Gordon Type Series for Capparelli's and G\"{o}llnitz-Gordon Identities}

\author[Kur\c{s}ung\"{o}z]{Ka\u{g}an Kur\c{s}ung\"{o}z}
\address{Faculty of Engineering and Natural Sciences, Sabanc{\i} University, \.{I}stanbul, Turkey}
\email{kursungoz@sabanciuniv.edu}

%    For articles to be published after 1 January 2010, you may use
%    the following version:


\date{August 2018}

\begin{abstract}
\noindent
We construct Andrews-Gordon type evidently positive series 
as generating functions for the partitions satisfying the difference conditions 
imposed by Capparelli's identities and G\"{o}llnitz-Gordon identities.  
The construction involves building base partitions, 
and move parts around subject to certain rules.  
Some missing cases are also considered, but no new identities are given.  
\end{abstract}

\maketitle

%%%%%%%%%%%%%%%%%%%%%%%%%%%%%%%%%%%%%%%%%%%%%%%%%%%%%%%%%%%
%%%%                                                   %%%%
%%%%                   Introduction                    %%%%
%%%%                                                   %%%%
%%%%%%%%%%%%%%%%%%%%%%%%%%%%%%%%%%%%%%%%%%%%%%%%%%%%%%%%%%%

\section{Introduction}
\label{secIntro}

The definition of an integer partition and the first partition identity 
is given by Euler \cite{Euler, Andrews-bluebook}.  
An integer partition $\lambda$ of a natural number $n$ is a non-decreasing sequence of positive integers 
that sum up to $n$.  
\begin{align*}
 n = & \lambda_1 + \lambda_2 + \cdots + \lambda_m , \\
 & \lambda_1 \leq \lambda_2 \leq \cdots \leq \lambda_m 
\end{align*}
The $\lambda_i$'s are called parts.  
The number of parts $m$ is called the length of the partition $\lambda$, denoted by $l(\lambda)$.  
The number being partitioned is the weight of the partition $\lambda$, denoted by $\vert \lambda \vert$.  
One could also reverse the weak inequalities and take non-decreasing sequences, 
but we will stick to this definition for the purposes of this note.  
The point is that reordering the same parts will not give us a new partition.  
For example, the five partitions of $n = 4$ are
\begin{align*}
 4, 1+3, 2+2, 1+1+2, 1+1+1+1.  
\end{align*}

Depending on context, we sometimes allow zeros to appear in the partition.  
Clearly, the zeros have no contribution to the weight of the partition, 
but the length changes as we add or take out zeros.  

\begin{theorem}[Euler's partition identity, \cite{Euler}]
 For $n \in \mathbb{N}$, 
 the number of partitions of $n$ into distinct parts 
 equals the number of partitions into odd parts.  
\end{theorem}

For example, there are two partitions of 4 into distinct parts (4, 1+3).  
There are also two partitions of 4 into odd parts (1+3, 1+1+1+1).  

The former condition in the theorem, namely parts being odd, 
is a \emph{congruence} condition.  
Parts $\equiv 1 \pmod{2}$ are allowed to appear, 
but parts $\equiv 0 \pmod{2}$ are not.  
The latter condition is a \emph{difference} condition.  
It simply says that the pairwise difference of parts are at least one.  
Or, successive parts must have difference at least one between them.  

There are a variety of proofs and versions of Euler's identity \cite{Andrews-bluebook}.  
We will focus on a particular one. 
\begin{align}
\label{gfEulerDist}
 & \sum_{n \geq 0} \frac{ q^{ (n^2+n)/2 } }{ (q; q)_n } \\ 
\label{gfEulerInclExcl}
 = & \left( (q^2; q^2)_\infty \right) \; \left( \frac{ 1 }{(q; q)_\infty } \right) \\
\label{gfEulerOdd}
 = & \frac{ 1 }{ (q; q^2)_\infty }
\end{align}
Here and throughout, for $n \geq 0$, 
\begin{align*}
 (a; q)_n = & \prod_{j = 1}^{n} (1 - a q^{j-1} ), \\
 (a; q)_\infty = \lim_{n \to \infty} (a; q)_n = & \prod_{j = 1}^{\infty} (1 - a q^{j-1} ).  
\end{align*}
The empty products are 1.  

Series such as \eqref{gfEulerDist} or products such as \eqref{gfEulerOdd} 
are called partition generating functions.  
The infinite product \eqref{gfEulerOdd} generates partitions into odd parts \cite{Andrews-bluebook}.  

It is also possible to interpret \eqref{gfEulerDist} as the 
generating function of partitions into distinct parts.  
The summation index $n$ keeps track of the number of parts.  
The exponent $(n^2+n)/2$ is recognized as $1 + 2 + \cdots + n$, 
which is the partition into $n$ distinct parts having minimal weight.  
We will call such partitions \emph{base partition}s.  
Finally, $\frac{1}{(q; q)_n}$ gives us a partition $\mu$ into $n$ parts, 
counting zeros.  
We incorporate parts of $\mu$ into the base partition 
by adding the $i$th largest part of $\mu$ to the $i$th largest part of the base partition.  
After this operation, the resulting partition still has distinct parts.  

Conversely, given a partition $\lambda$ into $n$ distinct parts, 
we can subtract $\mu_i$ from the $i$th smallest part 
so as to make the partition into the base partition $1+2+\cdots+n$, 
obtaining a partition into $n$ parts (counting zeros) along the way.  

The product of infinite products \eqref{gfEulerInclExcl} is important in its own right.  
If we expand the two infinite products separately as series, 
the first one will have negative coefficients.  
The product of the two series will eventually have positive coefficients, 
guaranteed by \eqref{gfEulerDist} or \eqref{gfEulerOdd}.  
However, it is not immediately clear that this is the case.  
This is an application of inclusion-exclusion method in integer partition theory \cite{Andrews-bluebook}.  

After Euler's introduction of integer partitions, and his partition identity, 
a milestone is the Rogers-Ramanujan identities \cite{RR}.  
We present Schur's version of the first identity \cite{Schur-RR}, 
followed by its original form.  

\begin{theorem}[the first Rogers-Ramanujan identity]
\label{thmRR1}
  For $n \in \mathbb{N}$, 
  the number of partitions of $n$ into distinct and non-consecutive parts 
  equals the number of partitions of $n$ into parts that are $\not\equiv 0, \pm 2 \pmod{5}$.  
\end{theorem}

For instance, $n = 4$ has two partitions into distinct and non-consecutive parts (4, 1+3); 
as well as two partitions into parts that are $\not\equiv 0, \pm 2 \pmod{5}$ (4, 1+1+1+1).  

Again, the former condition in the Theorem \ref{thmRR1} is a difference condition.  
It stipulates that the difference of successive parts are at least two.  
The latter condition on parts is obviously a congruence condition.  
The $q$-series, or original, version of Theorem \ref{thmRR1} is as follows.  
\begin{align}
 \sum_{n \geq 0} \frac{ q^{ n^2 } }{ (q; q)_n } 
 = \frac{ 1 }{ (q; q^5)_\infty (q^4; q^5)_\infty }
\end{align}
The right hand side generates partitions into parts $\not\equiv 0, \pm 2 \pmod{5}$.  
On the left hand side, once we agree that the summation index keeps track of the number of parts, 
and that $n^2 = 1+3+\cdots+(2n-1)$ is our base partition, 
it is easy to see that partitions into distinct and non-consecutive parts are generated.  

Perhaps the most seminal generalization of Rogers-Ramanujan identities are due to Gordon \cite{RRG}.  
\begin{theorem}
\label{thmGordon}
  For $n \in \mathbb{N}$, 
  let $B_{k, a}(n)$ be the number of partitions of $n$ in the form $\lambda_1 + \lambda_2 + \cdots + \lambda_m$
  where at most $a-1$ occurrences of the part 1 are allowed, 
  and $\lambda_{i+k-1} - \lambda_{i} \geq 2$.  
  Let $A_{k, a}(n)$ be the number of partitions of $n$ into parts 
  $\not\equiv 0, \pm a \pmod{(2k+1)}$.  
  Then $A_{k,a}(n) = B_{k,a}(n)$.  
\end{theorem}

Observe that the case $k = a = 2$ is the first Rogers-Ramanujan identity.  
$k = 2$ and $a = 1$ gives the second Rogers-Ramanujan identity.  
It is an easy exercise to write a generating function for $A_{k,a}(n)$.  
\begin{align}
\label{eqAndrewsGordonProduct}
  \sum_{n \geq 0} A_{k, a}(n) q^n
  = \prod_{ \substack{n \geq 1 \\ n \not\equiv 0, \pm a \pmod{(2k+1)} } } \frac{1}{1 - q^n}
\end{align}

Gordon gave a combinatorial proof of Theorem \ref{thmGordon}. 
Andrews, who independently discovered the same result \cite{Andrews-Kulliyat}, 
gave an algebraic proof \cite{Andrews-RRG}.  
Andrews defined $b_{k, a}(n, m)$ as the number of partitions enumerated by $B_{k, a}(n)$ 
which have $m$ parts, then showed that 
\begin{align*}
 & \sum_{m, n \geq 0} b_{k, a}(n, m) q^n x^m \\
 & = \sum_{n \geq 0} 
 \frac{ (-1)^n x^{kn} q^{(2k+1)(n^2+n)/2 - an} }{ (q; q)_n (xq^{n+1}; q)_\infty }
 - \frac{ (-1)^n x^{kn+a} q^{(2k+1)(n^2+n)/2 + a(n+1)} }{ (q; q)_n (xq^{n+1}; q)_\infty }.  
\end{align*}
It is not immediately clear that the right hand side produces 
positive integers when we compute its Taylor coefficients.  
Andrews later gave another interpretation for this series using inclusion-exclusion \cite{Andrews-RRG-Incl-Excl}.  

Almost a decade passed before Andrews announced another generating function 
for $b_{k, a}(n, m)$ with evidently positive coefficients \cite{Andrews-Gordon}.  
\begin{align}
\label{eqAndrewsGordonSeries}
  \sum_{m, n \geq 0} b_{k, a}(n, m) q^n x^m 
  = \sum_{n_1, n_2, \ldots, n_{k-1} \geq 0} 
    \frac{ q^{N_1^2 + N_2^2 + \cdots N_{k-1}^2 + N_{a} + N_{a+1} + \cdots + N_{k-1} } 
	  x^{N_1 + N_2 + \cdots + N_{k-1}} }
      { (q; q)_{n_1} (q; q)_{n_2} \cdots (q; q)_{n_{k-1} } }
\end{align}
Above, $N_r = n_r + n_{r+1} + \cdots + n_{k-1}$.  
Notice that substituting $x = 1$ clears the track of the number of parts.  
In other words, 
\begin{align*}
 \left. \sum_{m, n \geq 0} b_{k, a}(n, m) q^n x^m  \right\vert_{x = 1} 
 = \sum_{n \geq 0} \left( \sum_{m \geq 0} b_{k, a}(n, m) \right) q^n 
 = \sum_{n \geq 0} B_{k, a}(n) q^n.  
\end{align*}
For $x = 1$, the identity the right hand side of \eqref{eqAndrewsGordonProduct} = 
the right hand side of \eqref{eqAndrewsGordonSeries} 
has since been called Andrews-Gordon identities.  

It took some more time to find a combinatorial construction 
of the multiple series \eqref{eqAndrewsGordonSeries}, 
which is due to Bressoud \cite{Bressoud-RRG-Interpret}.  
Another construction was given in \cite{K-GordonMarking}.  

In 1988, when writing his PhD thesis on the representations of vertex operator algebras \cite{Capparelli-thesis}, 
Capparelli stumbled upon a pair of curious partition identities \cite{Capparelli-Conj}.  

\begin{theorem}[the first Capparelli identity]
\label{thmCapparelli1}
  For $n \in \mathbb{N}$, 
  let $CP_1(n)$ be the number of partitions of $n$ into parts 
  into parts that are at least two, 
  with difference at least two, 
  and difference at least four unless the sum of successive parts is a multiple of three.  
  Let $CPR_1(n)$ be the number of partitions of $n$ into distinct parts 
  which are $\not\equiv \pm 1 \pmod{6}$.  
  Then $CP_1(n) = CPR_1(n)$.  
\end{theorem}

\begin{theorem}[the second Capparelli identity]
\label{thmCapparelli2}
  For $n \in \mathbb{N}$, 
  let $CP_2(n)$ be the number of partitions of $n$ into parts 
  into parts that are not equal to two, 
  with difference at least two, 
  and difference at least four unless the sum of successive parts is a multiple of three.  
  Let $CPR_2(n)$ be the number of partitions of $n$ into distinct parts 
  which are $\not\equiv \pm 2 \pmod{6}$.  
  Then $CP_2(n) = CPR_2(n)$.  
\end{theorem}

The first Capparelli identity was proved by Andrews \cite{Andrews-Capparelli}.  
The identities now have other proofs such as \cite{Tamba-Xie}, 
including one by Capparelli himself \cite{Capparelli-Conj-Proven}.  

Alladi, Andrews and Gordon gave a generating function for $CP_1(n)$ \cite{AAG-Capparelli} with a refinement, 
but the combinatorial interpretation is via several transformations.  
Their series specialize to
\begin{align*}
  \sum_{n \geq 0} CP_1(n) q^n
  = (-q^4; q^6)_\infty 
  (-q^2; q^6)_\infty
  \sum_{n \geq 0} \frac{q^{6n^2 - 3n} }{ (q^3; q^3)_{2n} }.  
\end{align*}

Sills found two double series as a generating function for $CP_1(n)$
using Bailey pairs \cite{Sills-Capparelli}, 
but the coefficients had negative contributions as well, 
calling for inclusion-exclusion.  Sills' series are 
\begin{align*}
 & \sum_{n \geq 0} CP_1(n) q^n \\
 & = \sum_{n \geq 0} \sum_{j = 0}^{2n} \frac{ q^{n^2} \left( \frac{n-j+1}{3} \right) }{ (q;q)_{2n-j} (q; q)_j } \\
 & = 1 + \sum_{\substack{ n, j, r \geq 0 \\ (n, j, r) \neq (0, 0, 0) } }
  \frac{ q^{3n^2 + \frac{9}{2}r^2 + 3j^2 + 6nj + 6rj - \frac{5}{2}r - j } 
      (q^3; q^3)_{2j+r-1} (1 + q^{2r+2j} ) (1 - q^{6r+6j} ) }
      { (q^3; q^3)_n (q^3; q^3)_r (q^3; q^3)_j (-1; q^3)_{j+1} (q^3; q^3)_{n+2r+2j}  }, 
\end{align*}
where $\left( \frac{\cdot}{\cdot} \right)$ is the Legendre symbol.  

We will denote the partitions enumerated by $CP_r(n)$ having $m$ parts by $cp_r(n, m)$ for $r = 1, 2$.  
In this note, we first construct generating functions for $cp_1(n, m)$ and $cp_2(n, m)$, 
with some missing cases in between in section \ref{secCapparelli}.  
We do not have partition identities for the missing cases, though.  
The multiple series have evidently positive coefficients, 
and they are in the form of \eqref{eqAndrewsGordonSeries}.  
The series \eqref{eqCapparelli1GenFunc} and \eqref{eqGFCp2} 
are also constructed by Kanade and Russell \cite{KR-newseries}.  
They used partial staircases following Dousse and Lovejoy as in \cite{Dousse-Lovejoy}.  

Capparelli's identities also resemble G\"{o}llnitz-Gordon identities \cite{Gollnitz-GG, Gordon-GG}.  

\begin{theorem}[the first G\"{o}llnitz-Gordon identity]
\label{thmGG1}
  For $n \in \mathbb{N}$, 
  let $D_{2,2}(n)$ be the number of partitions of $n$ 
  in which the pairwise difference of parts is at least two, 
  and difference at least three unless both parts are odd.  
  Let $C_{2,2}(n)$ be the number of partitions of $n$ into parts
  that are $\equiv 1, 4, 7 \pmod{8}$.  
  Then, $C_{2,2}(n) = D_{2,2}(n)$
\end{theorem}

\begin{theorem}[the second G\"{o}llnitz-Gordon identity]
\label{thmGG2}
  For $n \in \mathbb{N}$, 
  let $D_{2,1}(n)$ be the number of partitions of $n$ 
  in which the pairwise difference of parts is at least two, 
  and difference at least three unless both parts are odd.  
  Let $C_{2,1}(n)$ be the number of partitions of $n$ into parts
  that are $\equiv 3, 4, 5 \pmod{8}$.  
  Then, $C_{2,1}(n) = D_{2,1}(n)$
\end{theorem}

We construct Andrews-Gordon like series for G\"{o}llnitz-Gordon identities as well, 
with some missing cases in section \ref{secGG}.  We do not have partition identities for the missing cases.  
The series is different from the standard ones \cite{Slater-GG}, 
or the double series given in \cite{AndrewsBringmannMahlburg}.  
\begin{align*}
 \sum_{n \geq 0} D_{2,2}(n) q^n 
 & = \sum_{n \geq 0} \frac{ q^{n^2} (-q; q^2)_n }{ (q^2; q^2)_n }, \\
 \sum_{n \geq 0} D_{2,1}(n) q^n 
 & = \sum_{n \geq 0} \frac{ q^{n(n+2)} (-q; q^2)_n }{ (q^2; q^2)_n }.  
\end{align*}
It is straightforward to account for the number of parts in the G\"{o}llnitz-Gordon partitions
and have 
\begin{align*}
 \sum_{n \geq 0} D_{2,2}(n, m) x^m q^n 
 & = \sum_{n \geq 0} \frac{ q^{n^2} (-q; q^2)_n x^n }{ (q^2; q^2)_n }, \\
 \sum_{n \geq 0} D_{2,1}(n, m) x^m q^n 
 & = \sum_{n \geq 0} \frac{ q^{n(n+2)} (-q; q^2)_n x^n }{ (q^2; q^2)_n }.  
\end{align*}

We collect the $q$-series identities versions of the presented results 
in the brief section \ref{secCorollaries}.  
We conclude with some directions for future research in section \ref{secConclusion}.  

\section{Capparelli's Identities and Some Missing Cases}
\label{secCapparelli}

\begin{theorem}[cf. the first Capparelli Identity]
\label{thmCapparelli1GenFunc}
  For $n, m \in \mathbb{N}$, 
  let $cp_1(n, m)$ denote the number of partitions of $n$ into $m$ parts such that 
  all parts are at least 2, the difference is at least 2 at distance 1 
  and it is at least 4 unless the successive parts add up to a multiple of 3.  
  Then, 
  \begin{equation}
  \label{eqCapparelli1GenFunc}
    \sum_{m, n \geq 0} cp_1(n, m) q^n x^m
    = \sum_{n_1, n_2 \geq 0} \frac{ q^{ 2 n_1^2 + 6 n_1 n_2 + 6 n_2^2 } x^{ n_1 + 2 n_2 } }
      { (q; q)_{n_1} (q^3; q^3)_{n_2} }
  \end{equation}
\end{theorem}

\begin{proof}
  We will show that each partition $\lambda$ counted by $cp_1(n, m)$ 
  corresponds to a triple of partitions $(\beta, \mu, \eta)$ 
  which are described as follows.  
  
  \begin{align}
    \label{eqDefBasePtn}
    \beta = & [2 + 4] + [8 + 10] + \cdots + [(6 n_2 - 4) + (6 n_2 - 2)] \\
    \nonumber
      + & (6 n_2 + 2) + (6 n_2 + 6) + \cdots + (6 n_2 + 4 n_1 - 2), 
  \end{align}
  namely the \emph{base partition} which has $n_2$ \emph{pairs} with difference 2, 
  and $n_1$ \emph{singletons} whose pairwise difference is at least 4. 
  The pairs are shown in brackets for convenience.  
  The number of parts of $\beta$ is $m = n_1 + 2 n_2$.  
  Observe that $\beta$ satisfies the conditions set forth by $cp_1(n, m)$.  
  The weight of $\beta$ is
  \[
    6 \left[1 + 3 + \cdots + (2n_2 - 1)\right] 
    + \left[ n_1(6 n_2 + 2) + 4 \binom{n_1}{2} \right] = 6n_2^2 + 6 n_2 n_1 + 2 n_1^2, 
  \]
  
  $\mu$ is a partition with $n_1$ parts, counting zeros;  
  and $\eta$ is a partition into multiples of 3 with $n_2$ parts, counting zeros.  
  
  At this point, it is clear that 
  
  \begin{equation}
  \label{eqRawGenFunc}
    \sum_{ \substack{n_1, n_2 \geq 0 \\ \mu, \eta} } 
      q^{ \vert \beta \vert + \vert \mu \vert + \vert \eta \vert }
      x^{ n_1 + 2 n_2 } 
    = \sum_{n_1, n_2 \geq 0} \frac{ q^{ 2 n_1^2 + 6 n_1 n_2 + 6 n_2^2 } x^{ n_1 + 2 n_2 } }
      { (q; q)_{n_1} (q^3; q^3)_{n_2} }
  \end{equation}
  
  We will obtain a unique $\lambda$ from $(\beta, \mu, \eta)$ by a series of forward moves, 
  and a unique $(\beta, \mu, \eta)$ from $\lambda$ by a series of backward moves.  
  Showing that the sequences of forward and backward moves are inverses of each other will yield 
  \begin{equation*}
    \sum_{m, n \geq 0} cp_1(n, m) q^n x^m
    = \sum_{ \substack{n_1, n_2 \geq 0 \\ \mu, \eta} } 
      q^{ \vert \beta \vert + \vert \mu \vert + \vert \eta \vert }
      x^{ n_1 + 2 n_2 }, 
  \end{equation*}
  completing the proof.  
  
  Parts of $\mu$ will pertain to singletons in a partition enumerated by $cp_1(n, m)$, 
  whose distances to preceding or suceeding parts are greater than or equal to 3, but not both 3.  
  
  Parts of $\eta$ will pertain to pairs in a partition enumerated by $cp_1(n, m)$, 
  a pair of parts differing by 2 or 3.  
  Let's keep in mind that we indicate pairs by putting brackets around them.    
  
  To resolve the potential overlap between pairs and singletons, 
  notice that the only ambiguity arises when a partition counted by $cp_1(n, m)$ 
  has a streak of consecutive multiples of 3:  
  \[
    (\textrm{parts } \leq 3k - 4), 
    3k, 3k+3, \ldots, 3k+3l, 
    (\textrm{parts } \geq 3k+3l+4).  
  \]
  If $l$ is odd, the pairs are obvious:  
  \[
    (\textrm{parts } \leq 3k - 4), 
    [3k, 3k+3], [3k+6, 3k+9], 
    \ldots
  \]
  \[
    \ldots, 
    [3k+3l-3, 3k+3l], 
    (\textrm{parts } \geq 3k+3l+4).  
  \]
  If $l$ is even and we are implementing forward moves on pairs, 
  we declare $3k$ as a singleton:  
  \[
    (\textrm{parts } \leq 3k - 4), 
    3k, [3k+3, 3k+6], [3k+9, 3k+12], 
    \ldots
  \]
  \[
    \ldots, 
    [3k+3l-1, 3k+3l], 
    (\textrm{parts } \geq 3k+3l+4).  
  \]
  Else if $l$ is even and we are not implementing forward moves on pairs, 
  we declare $3k+3l$ as a singleton:  
  \[
    (\textrm{parts } \leq 3k - 4), 
    [3k, 3k+3], [3k+6, 3k+9], 
    \ldots
  \]
  \[
    \ldots, 
    [3k+3l-6, 3k+3l-3], 3k+3l, 
    (\textrm{parts } \geq 3k+3l+4).  
  \]
  This way, we will be able to mention \emph{the $k$th largest (or smallest)} pair (or singleton) 
  without confusion.  
  The nature of moves, forward or backward, will be clear from context.  
  
  Given $(\beta, \mu, \eta)$,  we first add the $i$th largest part of $\mu$ 
  to the $i$th largest singleton in $\beta$ for $i = 1, 2, \ldots, n_1$ in this order.  
  These are the forward moves on singletons.  
  The pairs are intact, and the singletons have pairwise difference greater than or equal to 4
  after the forward moves on singletons, 
  so the difference condition given by $cp_1(n, m)$ is preserved.  
  
  Then, we move the $i$th largest pair in the resulting partition
  $\frac{1}{3}\cdot$ (the $i$th largest part of $\eta$) times forward 
  for $i = 1, 2, \ldots, n_2$, in this order.  
  One forward move will add 3 to the sum of the parts in a pair.  
  
  We now define the forward moves on pairs.  There are several cases.  
  Here and elsewhere, 
  we indicate the pair or singleton being moved in boldface.  
  
  {\bf (case Ia)}
  \begin{align*}
    (\textrm{parts } \leq 3k-5), & [\mathbf{3k-1, 3k+1}], (\textrm{parts } \geq 3k+6) \\[3pt]
    & \big\downarrow \textrm{ one forward move} \\[3pt]
    (\textrm{parts } \leq 3k-5), & [\mathbf{3k, 3k+3}], (\textrm{parts } \geq 3k+6) 
  \end{align*}
  Here, there is a potential regrouping for determining pairs if there is a $3k+6$.  
  
  {\bf (case IIa)}
  \begin{align*}
    (\textrm{parts } \leq 3k-3), & [\mathbf{3k, 3k+3}], (\textrm{parts } \geq 3k+8) \\[3pt]
    & \big\downarrow \textrm{ one forward move} \\[3pt]
    (\textrm{parts } \leq 3k-3), & [\mathbf{3k+2, 3k+4}], (\textrm{parts } \geq 3k+8) 
  \end{align*}
  
  {\bf (case Ib)}
  \begin{align*}
    (\textrm{parts } \leq 3k-5), & [\mathbf{3k-1, 3k+1}], 3k+5, (\textrm{parts } \geq 3k+9) \\[3pt]
    & \big\downarrow \textrm{ one forward move} \\[3pt]
    (\textrm{parts } \leq 3k-5), & [\mathbf{3k}, \underbrace{\mathbf{3k+3}], 3k+5}_{!}, (\textrm{parts } \geq 3k+9) 
    \textrm{ (temporarily)} \\[3pt]
    & \big\downarrow \textrm{ adjustment} \\[3pt]
    (\textrm{parts } \leq 3k-5), & 3k-1, [\mathbf{3k+3, 3k+6}], (\textrm{parts } \geq 3k+9) 
  \end{align*}
  Here again, there is a potential regrouping of pairs if there is a $3k+9$.  
  
  {\bf (case IIb)}
  \begin{align*}
    (\textrm{parts } \leq 3k-3), & [\mathbf{3k, 3k+3}], 3k+7, (\textrm{parts } \geq 3k+11) \\[3pt]
    & \big\downarrow \textrm{ one forward move} \\[3pt]
    (\textrm{parts } \leq 3k-3), & [\mathbf{3k+2}, \underbrace{\mathbf{3k+4}], 3k+7}_{!}, (\textrm{parts } \geq 3k+11) 
    \textrm{ (temporarily)} \\[3pt]
    & \big\downarrow \textrm{ adjustment} \\[3pt]
    (\textrm{parts } \leq 3k-3), & 3k+1, [\mathbf{3k+5, 3k+7}], (\textrm{parts } \geq 3k+11) 
  \end{align*}
  
  It is prudent to remind the reader again that any forward move on a pair 
  increases the weight of the partition by 3.  
  The adjustments do not alter the weight.  
  To see that the difference conditions are retained, 
  one simply checks the difference between successive parts.  
  
  Because the $i$th largest pair is moved 
  $\frac{1}{3}\cdot$($i$th largest part of $\eta$) times forward, 
  a larger pair is moved forward at least as many times as the smaller pair preceding it.  
  One can incorporate immediately preceding pairs to the above cases 
  and show that one forward move on the larger pair allows one forward move of the preceding smaller pair.  
  
  For example, in case {\bf (Ib)}, the immediately preceding pair could have been $[3k-5, 3k-5]$.  
  
  \begin{align*}
    (\textrm{parts } \leq 3k-11), & [3k-7, 3k-5], [\mathbf{3k-1, 3k+1}], 3k+5, (\textrm{parts } \geq 3k+9) \\[3pt]
    & \big\downarrow \textrm{ one forward move on the larger pair and adjustment} \\[3pt]
    (\textrm{parts } \leq 3k-11), & [\mathbf{3k-7, 3k-5}], 3k-1, [3k+3, 3k+6], (\textrm{parts } \geq 3k+9) \\[3pt]
    & \big\downarrow \textrm{ one forward move on the smaller pair and adjustment} \\[3pt]
    (\textrm{parts } \leq 3k-11), & 3k-7, [\mathbf{3k-3, 3k}], 3k-1, [3k+3, 3k+6], (\textrm{parts } \geq 3k+9), 
  \end{align*}
  followed by a potential regrouping of pairs.  
  
  Cases {\bf (Ia)}, {\bf (IIa)}, {\bf (Ib)}, and {\bf (IIb)} exhaust all possibilities.  
  A seemingly forgotten one is 
  \[
    (\textrm{parts } \leq 3k-3), [\mathbf{3k, 3k+3}], 3k+6, (\textrm{parts } \geq 3k+9).  
  \]
  However, if there is a $3k+9$ as well, this means that the pairs $[3k, 3k+3]$ and $[3k+6, 3k+9]$
  had been moved forward an equal number of times, 
  so the smaller pair is not allowed an extra move.  
  Or the pairing stipulates that $[3k+3, 3k+6]$ is a pair and $3k$ is a singleton.  
  
  We have now produced a partition $\lambda$ which is enumerated by $cp_1(n, m)$.  
  
  Conversely, given a partition enumerated by $cp_1(n, m)$, we first determine the pairs and singletons.  
  The pairs are the pair of parts whose pairwise distance is 2 or 3.  
  If there is a streak of multiples of 3, pair them beginning at the smaller end, 
  so the largest of them is left out.  Call the number of pairs $n_2$.  
  The remaining $n_1$ parts are singletons.  
  Since $\lambda$ satisfies the condition set forth by $cp_1(n, m)$, 
  the sum of pairs is a multiple of 3, 
  and any singleton has distance greater than or equal to 3 
  to its preceding and succeding parts (not both equal 3).  
  
  Start with the smallest pair $[3k+2 ,3k+4]$ or $[3k, 3k+3]$, 
  and perform $\frac{\eta_1}{3}$ backward moves on it until it becomes $[2, 4]$, 
  thus determining $\eta_1$, the smallest part of $\eta$.  
  If the smallest pair is already $[2, 4]$, set $\eta_1 = 0$.  
  
  The backward moves will fall in one of the following cases.  
  
  {\bf (case I'a)}
  \begin{align*}
    (\textrm{parts } \leq 3k-5), & [\mathbf{3k, 3k+3}], (\textrm{parts } \geq 3k+6) \\[3pt]
    & \big\downarrow \textrm{ one backward move} \\[3pt]
    (\textrm{parts } \leq 3k-5), & [\mathbf{3k-1, 3k+1}], (\textrm{parts } \geq 3k+6) 
  \end{align*}
  
  {\bf (case II'a)}
  \begin{align*}
    (\textrm{parts } \leq 3k-3), & [\mathbf{3k+2, 3k+4}], (\textrm{parts } \geq 3k+8) \\[3pt]
    & \big\downarrow \textrm{ one backward move} \\[3pt]
    (\textrm{parts } \leq 3k-3), & [\mathbf{3k, 3k+3}], (\textrm{parts } \geq 3k+8), 
  \end{align*}
  with a potential regrouping of pairs if there is a $3k-3$.  
  
  {\bf (case I'b)}
  \begin{align*}
    (\textrm{parts } \leq 3k-5), & 3k-1, [\mathbf{3k+3, 3k+6}], (\textrm{parts } \geq 3k+9) \\[3pt]
    & \big\downarrow \textrm{ one backward move} \\[3pt]
    (\textrm{parts } \leq 3k-5), & \underbrace{3k-1, [\mathbf{3k+2} }_{!}, \mathbf{3k+4}], (\textrm{parts } \geq 3k+9) 
    \textrm{ (temporarily)} \\[3pt]
    & \big\downarrow \textrm{ adjustment} \\[3pt]
    (\textrm{parts } \leq 3k-5), & [\mathbf{3k-1, 3k+1}], 3k+5, (\textrm{parts } \geq 3k+9) 
  \end{align*}
  
  {\bf (case II'b)}
  \begin{align*}
    (\textrm{parts } \leq 3k-3), & 3k+1, [\mathbf{3k+5, 3k+7}], (\textrm{parts } \geq 3k+11) \\[3pt]
    & \big\downarrow \textrm{ one backward move} \\[3pt]
    (\textrm{parts } \leq 3k-3), & \underbrace{3k+1, [\mathbf{3k+3} }_{!}, \mathbf{3k+6}] (\textrm{parts } \geq 3k+11) 
    \textrm{ (temporarily)} \\[3pt]
    & \big\downarrow \textrm{ adjustment} \\[3pt]
    (\textrm{parts } \leq 3k-3), & [\mathbf{3k, 3k+3}], 3k+7, (\textrm{parts } \geq 3k+11) 
  \end{align*}
  Here again, there is a potential regoruping of pairs if there is a $3k-3$.  
  
  Notice that the only hurdles ahead of the smallest pair are the singletons, 
  which can be tackled by cases {\bf (I'b)} and {\bf (II'b)}.  
  Analogous to the implementation of the forward moves, 
  any backward move on a pair allows a backward move on the immediately succeeding pair.  
  
  Once the smallest pair is stowed as $[2, 4]$ 
  and the required number of moves is recorded as $\frac{\eta_1}{3}$, 
  we continue with the next smallest pair.  
  We perform $\frac{\eta_2}{3}$ backward moves on it until it becomes $[8, 10]$, 
  thus determining $\eta_2$ and so on.  
  The previous paragraph shows that .  
  \[
    \eta_1 \leq \eta_2 \leq \cdots \leq \eta_{n_2}.  
  \]
  
  When all backward moves are performed, the intermediate partition looks like 
  \[
    [2, 4], [8, 10], \ldots, [6n_2 - 4, 6n_2 - 2], \textrm{ singletons } \geq 6n_2 + 2
  \]
  and we have a partition $\eta$ with at most $n_2$ parts into multiples of 3.  
  At this point, the distances between singletons are at least 4, 
  because distance 3 indicates a pair which is not accounted for.  
  
  To make the singletons 
  \[
    6n_2 + 2, 6n_2 + 6, \ldots, 6n_2 + 4n_1 - 2, 
  \]
  we subtract $\mu_1$ from the smallest singleton, $\mu_2$ from the next smallest, and so on.  
  This  will give us the base partition $\beta$, 
  and the partition $\mu$ into at most $n_1$ parts.  
  It is evident that 
  \[
    \mu_1 \leq \mu_2 \leq \cdots \leq \mu_{n_1}.  
  \]
  
  To finish the proof, we notice that the corresponding cases for the backward and forward moves 
  have switched inputs and outputs.  
  The forward moves are performed on singletons first, from the largest to smallest, 
  using parts of $\mu$ from largest to smallest, respectively; 
  then on pairs next, from largest to smallest, using $\frac{1}{3}\times$ parts of $\eta$, 
  from largest to smallest, respectively.  
  The backward moves are performed in the exact reverse order.  
  Therefore, the $\lambda$'s enumerated by $cp_1(n, m)$ 
  are in 1-1 correspondence with the triples $(\beta, \mu, \eta)$, as desired.  
\end{proof}

It is possible to make the backward and forward moves \emph{exact} opposites, 
including the rulebreaking in the middle.  
However, we prefer the more intuitive versions described in the proof.  

% begin example 

 {\bf Example: }
 With the notation in the proof of Theorem \ref{thmCapparelli1GenFunc}, 
 Let $\beta$ be the base partition with 2 pairs and 2 singletons, 
 $\mu = 1+2$, and $\eta = 3+9$.  
 \begin{align*}
  \beta = [2, 4], [8, 10], 14, \mathbf{18}
 \end{align*}
 The weight of $\beta$ is 56.  
 After incorporating parts of $\mu$ as forward moves on the singletons, 
 the intermediate partition becomes
 \begin{align*}
  [2, 4], [\mathbf{8, 10}], 15, 20.  
 \end{align*}
  We now perform $\frac{1}{3}\eta_2 = 3$ 
  forward moves on the larger pair $[8, 10]$.  
  \begin{align*}
   \big\downarrow \textrm{ the first forward move on the larger pair }
  \end{align*}
 \begin{align*}
  [2, 4], [\mathbf{9, 12}], 15, 20  
 \end{align*}
  \begin{align*}
   \big\downarrow \textrm{ regrouping the pairs }
  \end{align*}
 \begin{align*}
  [2, 4], 9, [\mathbf{12, 15}], 20  
 \end{align*}
  \begin{align*}
   \big\downarrow \textrm{ the second forward move on the larger pair }
  \end{align*}
 \begin{align*}
  [2, 4], 9, [\mathbf{14, 16}], 20  
 \end{align*}
  \begin{align*}
   \big\downarrow \textrm{ the third, and the last, forward move on the larger pair }
  \end{align*}
 \begin{align*}
  [2, 4], 9, [\mathbf{15}, \underbrace{\mathbf{18}], 20 }_{!} 
 \end{align*}
  \begin{align*}
   \big\downarrow \textrm{ adjustment }
  \end{align*}
 \begin{align*}
  [\mathbf{2, 4}], 9, 14, [18, 21] 
 \end{align*}
 We now perform $\frac{1}{3}\eta_1 = 1$ forward moves on the smaller pair.  
  \begin{align*}
   \big\downarrow \textrm{ one forward move on the smaller pair }
  \end{align*}
 \begin{align*}
  [\mathbf{3, 6}], 9, 14, [18, 21] 
 \end{align*}
  \begin{align*}
   \big\downarrow \textrm{ regrouping pairs }
  \end{align*}
 \begin{align*}
  3, [\mathbf{6, 9}], 14, [18, 21] 
 \end{align*}
 This final partition is $\lambda$.  
 Its weight is $71 = \vert \lambda \vert = \vert \beta \vert + \vert \mu \vert + \vert \eta \vert = 
 56 + 3 + 12$ indeed.  
 
%   end example 

\begin{theorem}[cf. the second Capparelli Identity]
\label{thmCapparelli2GenFunc}
  For $n, m \in \mathbb{N}$, 
  let $cp_2(n,m)$ be the number of partitions of $n$ into $m$ parts 
  with no occurrences of the part 2, 
  such that the pairwise difference is at least two, 
  and at least four unless the successive parts add up to a multiple of three.  
  Then, 
  \begin{align}
  \label{eqGFCp2}
    \sum_{m, n \geq 0} cp_2(m,n) q^n x^m 
    = \sum_{n_1, n_2 \geq 0} \frac{ q^{6n_2^2 + 3n_2 + 2n_1^2 + n_1 + 6n_2 n_1} (1 + xq^{3n_2 + 2n_1 + 1}) x^{2n_2 + n_1} }
      { (q; q)_{n_1} (q^3; q^3)_{n_2} }
  \end{align}
\end{theorem}

\begin{proof}
 The proof is similar to that of Theorem \ref{thmCapparelli1GenFunc}, 
 except that one needs to distinguish the cases in which a partition $\lambda$ 
 enumerated by $cp_2(n, m)$ contains a 1 or not.  
 The respective base partitions are 
 \begin{align}
 \label{baseptncp2beta1}
  \beta = 1, [5, 7], [11, 13], \ldots [6n_2 - 1, 6n_2 + 1], 6n_2 + 5, 6n_2 + 9, \ldots , 6n_2 + 4n_1 + 1
 \end{align}
 with weight $6n_2^2 + 6n_2 + 2n_1^2 + 3n_1 + 6n_2 n_1 + 1$, or
 \begin{align}
 \label{baseptncp2beta2}
  \beta = [3, 6], [9, 12], \ldots [6n_2 - 3, 6n_2], 6n_2 + 3, 6n_2 + 7, \ldots , 6n_2 + 4n_1 - 1 
 \end{align}
 with weight $6n_2^2 + 3n_2 + 2n_1^2 + n_1 + 6n_2 n_1$.  
 
 In \eqref{baseptncp2beta1}, the singleton 1 is not moved at all, but the other $n_1$ singletons may be.  
 In either \eqref{baseptncp2beta1} or \eqref{baseptncp2beta2}, the pairs are indicated in brackets.  
 The respective generating functions using the base partitions \eqref{baseptncp2beta1} or \eqref{baseptncp2beta2} 
 are constructed as in the proof of Theorem \ref{thmCapparelli1GenFunc}.  
 
 Observe that the generating function for \eqref{baseptncp2beta1} has an extra factor $xq$ 
 due to the presence of the unmoved part 1.  
 The sum of the obtained series gives the result.  
\end{proof}

% begin example

  {\bf Example: } 
  Using the notation in the proof of Theorem \ref{thmCapparelli2GenFunc}, 
  let's take 
  \begin{align*}
   \lambda = [\mathbf{3, 6}], 9, 14, [18, 21]  
  \end{align*}
  The pair $[3, 6]$ is chosen over $[6, 9]$, 
  because we will not be performing forward moves on pairs.  
  This partition counted by $cp_2(71, 6)$ has no 1's, 
  therefore, we will be aiming at the base partition \eqref{baseptncp2beta2}. 
  The smaller pair is already in its terminal place, so we set $\eta_1 = 0$.  
  To determine $\eta_2$, we move the larger pair backwards and record the number of moves.  
  \begin{align*}
    \big\downarrow \textrm{ one backward move on the larger pair }
  \end{align*}
  \begin{align*}
   [3, 6], 9, \underbrace{14, [\mathbf{17}}_{!}, \mathbf{19}]  
  \end{align*}
  \begin{align*}
    \big\downarrow \textrm{ adjustment }
  \end{align*}
  \begin{align*}
   [3, 6], 9, [\mathbf{14, 16}], 20  
  \end{align*}
  \begin{align*}
    \big\downarrow \textrm{ another backward move on the larger pair }
  \end{align*}
  \begin{align*}
   [3, 6], 9, [\mathbf{12, 15}], 20  
  \end{align*}
  \begin{align*}
    \big\downarrow \textrm{ regrouping pairs }
  \end{align*}
  \begin{align*}
   [3, 6], [9, 12], \mathbf{15}, 20  
  \end{align*}
  We made two backward moves on the larger pair to stow it in its place.  
  We now know that $\eta = 0 + 6$.  
  Recall that one backward move on a pair drops the weight by three.  
  
  Finally, we see that the smaller singleton is already as small as it can be.  
  The larger can be moves backward one time, though.  
  Therefore, $\mu = 0 + 1$, and we have 
  \begin{align*}
   \beta = [3, 6], [9, 12], 15, \mathbf{19},   
  \end{align*}
  $\mu = 0+1$, and $\eta = 0+6$.  
  The weight of $\beta$ is 64.  
  
% end example

It is possible to alter the difference conditions in Capparelli's identities, 
and write generating functions for the resulting classes of partitions.  
Those generating functions may or may not give rise to partition identities, though.  
In other words, it may or may not be possible to express the constructed series as nice infinite products.  
We give several examples below, with hints for their proofs.  
The proofs are similar to proofs of Theorems \ref{thmCapparelli1GenFunc} and \ref{thmCapparelli2GenFunc}.  

\begin{theorem}
\label{thmcp0}
  For $n, m \in \mathbb{N}$, 
  let $cp_0(n, m)$ be the number of partitions of $n$ into $m$ parts 
  with pairwise difference of parts is at least two, 
  and it is at least four unless the successive parts add up to a multiple of three.  
  Then, 
  \begin{align*}
%   \label{eqGFCp0}
    \sum_{m, n \geq 0} cp_0(m,n) q^n x^m 
    = \sum_{n_1, n_2 \geq 0} \frac{ q^{6n_2^2 + 2n_1^2 + 6n_2 n_1} (1 + xq^{6n_2 + 3n_1 + 1}) x^{2n_2 + n_1} }
      { (q; q)_{n_1} (q^3; q^3)_{n_2} }
  \end{align*}
\end{theorem}

\begin{proof}
 We follow the proof of Theorem \ref{thmCapparelli2GenFunc}, but use the base partitions 
 \begin{align}
 \label{baseptncp0beta1}
  \beta = 1, [5, 7], [11, 13], \ldots [6n_2 - 1, 6n_2 + 1], 6n_2 + 5, 6n_2 + 9, \ldots , 6n_2 + 4n_1 + 1
 \end{align}
 with weight $6n_2^2 + 6n_2 + 2n_1^2 + 3n_1 + 6n_2 n_1 + 1$, or
 \begin{align}
 \label{baseptncp0beta2}
  \beta = [2, 4], [8, 10], \ldots [6n_2 - 4, 6n_2-2], 6n_2 + 2, 6n_2 + 6, \ldots , 6n_2 + 4n_1 - 2 
 \end{align}
 with weight $6n_2^2 + 2n_1^2 + 6n_2 n_1$.  
 We use \eqref{baseptncp0beta1} for partitions containing a 1.  That 1 is not moved throughout.  
\end{proof}

\begin{theorem}
\label{thmcp1-1}
  For $n, m \in \mathbb{N}$, 
  let $cp_{1-1}(n, m)$ be the number of partitions of $n$ into $m$ parts 
  with pairwise difference of parts is at least two, 
  and it is at least four unless the successive parts add up to $\equiv 1 \pmod{3}$.  
  Then, 
  \begin{align*}
%   \label{eqGFCp1-1}
    \sum_{m, n \geq 0} cp_{1-1}(m,n) q^n x^m 
    = \sum_{n_1, n_2 \geq 0} \frac{ q^{6n_2^2 - 2n_2 + 2n_1^2 - n_1 + 6n_2 n_1} x^{2n_2 + n_1} }
      { (q; q)_{n_1} (q^3; q^3)_{n_2} }
  \end{align*}
\end{theorem}

\begin{proof}
 The admissible pairs in partitions enumerated by $cp_{1-1}(n, m)$ 
 are of the form $[3k+1, 3k+3]$ or $[3k-1, 3k+2]$.  
 The proof is complete once we see that $cp_{1-1}(n+m, m) = cp_1(n, m)$, 
 and use Theorem \ref{thmCapparelli1GenFunc}.  
\end{proof}

\begin{theorem}
\label{thmcp1-2}
  For $n, m \in \mathbb{N}$, 
  let $cp_{1-2}(n, m)$ be the number of partitions of $n$ into $m$ parts 
  with pairwise difference of parts is at least two, 
  and it is at least four unless the successive parts add up to $\equiv 2 \pmod{3}$.  
  Then, 
  \begin{align*}
%   \label{eqGFCp1-2}
    \sum_{m, n \geq 0} cp_{1-2}(m,n) q^n x^m 
    = \sum_{n_1, n_2 \geq 0} \frac{ q^{6n_2^2 - n_2 + 2n_1^2 - n_1 + 6n_2 n_1} x^{2n_2 + n_1} }
      { (q; q)_{n_1} (q^3; q^3)_{n_2} }
  \end{align*}
\end{theorem}

{\bf Remark: } Notice that there are no restrictions on which parts can occur 
in Theorems \ref{thmcp0}-\ref{thmcp1-2}.  

\begin{proof}
 The admissible pairs in partitions enumerated by $cp_{1-2}(n, m)$ 
 are of the form $[3k, 3k+2]$ or $[3k-2, 3k+1]$.  
 We imitate the proof of Theorem \ref{thmCapparelli1GenFunc} using the base partition
 \begin{align*}
%  \label{baseptncp1-2beta}
  \beta = [1, 4], [7, 10], \ldots [6n_2 - 5, 6n_2-2], 6n_2 + 1, 6n_2 + 5, \ldots , 6n_2 + 4n_1 - 3,  
 \end{align*}
 whose weight is $6n_2^2 - n_2 + 2n_1^2 - n_1 + 6n_2 n_1$.  
 This weight is minimal among those partitions having $n_2$ pairs and $n_1$ singletons.  
\end{proof}

One can populate this list of theorems by imposing more conditions on the smallest parts, 
or on which parts can occur.  The proofs will all be alike.  

\section{G\"{o}llnitz-Gordon Identities and Some Missing Cases}
\label{secGG}

The series in this section appear to be new.  

\begin{theorem}[cf. G\"{o}llnitz-Gordon Identities]
\label{thmGG1-2GenFunc}
  For $n, m \in \mathbb{N}$, 
  let $D_{2,2}(n, m)$ be the number of partitions of $n$ into $m$ parts 
  in which the pairwise difference of parts is at least two, 
  and at least three unless the successive parts are both odd.  
  Let $D_{2, 1}(n, m)$ be the partitions enumerated by $D_{2,2}(n, m)$
  in which the smallest part is at least three.  
  Then, 
  \begin{align}
  \label{eqGfGG1}
    \sum_{m, n \geq 0} D_{2,2}(n, m) q^n x^m 
    = \sum_{n_1, n_2 \geq 0} \frac{ q^{4n_2^2 + (3n_1^2 - n_1)/2 + 4n_2 n_1} x^{2n_2 + n_1} }
      { (q; q)_{n_1} (q^4; q^4)_{n_2} } \\
  \label{eqGfGG2}
    \sum_{m, n \geq 0} D_{2,1}(n, m) q^n x^m 
    = \sum_{n_1, n_2 \geq 0} \frac{ q^{4n_2^2 + 4n_2 + (3n_1^2 + 3n_1)/2 + 4n_2 n_1} x^{2n_2 + n_1} }
      { (q; q)_{n_1} (q^4; q^4)_{n_2} } 
  \end{align}
\end{theorem}

\begin{proof}
 \eqref{eqGfGG2} follows from \eqref{eqGfGG1} once we see that $D_{2,1}(n,m) = D_{2,2}(n+2m,m)$.  
 The proof of \eqref{eqGfGG1} has the same framework as the proof of Theorem \ref{thmCapparelli1GenFunc}.  
 The difference is that the pairs are determined from parts with difference exactly two, 
 and the remaining parts are singletons.  
 A singleton may have difference two with its predecessor or successor, but not both.  
 The base partition is
 \begin{align*}
%  \label{baseptngg1-2beta}
  [1, 3], [5, 7], \ldots, [4n_2 - 3, 4n_2 - 1], 4n_2 + 1, 4n_2 + 4, \ldots, 4n_2 + 3 n_1 - 2,  
 \end{align*}
 the weight of which is $4n_2^2 + (3n_1^2 - n_1)/2 + 4n_2 n_1$.  
 This is the partition having $n_2$ pairs, $n_1$ singletons, and minimal weight.  
 The pairs are indicated in brackets.  
 They are the pair of parts with difference two, hence both parts in them are necessarily odd.  
 Of two pairs, the one with larger parts is declared the greater pair.  
 
 The forming of pairs may be updated dynamically during forward or backward moves.  
 When forming pairs in a successive streak of odd parts, 
 we use the largest pair first in case of implementing forward moves on pairs, 
 and smallest pair first otherwise.  
 These cases differ only when the number of odd parts in the consecutive streak is odd and greater than three.  
 
 The forward moves on the pairs are performed according to the exclusive cases below.  
  \begin{align*}
    (\textrm{parts } \leq 2k-3), & [\mathbf{2k-1, 2k+1}], (\textrm{parts } \geq 2k+5) \\[3pt]
    & \big\downarrow \textrm{ one forward move on the displayed pair} \\[3pt]
    (\textrm{parts } \leq 2k-3), & [\mathbf{2k+1, 2k+3}], (\textrm{parts } \geq 2k+5) 
  \end{align*}
  with a possible update of pairs if there is a part $= 2k+5$
  
  \begin{align*}
    (\textrm{parts } \leq 2k-3), & [\mathbf{2k-1, 2k+1}], 2k+4, (\textrm{parts } \geq 2k+7) \\[3pt]
    & \big\downarrow \textrm{ one forward move on the displayed pair} \\[3pt]
    (\textrm{parts } \leq 2k-3), & [\mathbf{2k+1}, \underbrace{\mathbf{2k+3}], 2k+4}_{!}, (\textrm{parts } \geq 2k+7) 
    \textrm{ (temporarily)} \\[3pt]
    & \big\downarrow \textrm{ adjustment} \\[3pt]
    (\textrm{parts } \leq 2k-3), & 2k, [\mathbf{2k+3, 2k+5}], (\textrm{parts } \geq 2k+7) 
  \end{align*}
  Here again, there is a possible update on the pairing if there is a part $= 2k+7$.  
  The adjustment does not alter the weight of the partition.  
  
  Each forward move on the pairs adds 4 to the weight, 
  hence the partition $\eta$ consists of multiples of four instead of multiples of three.  
  Backward moves can be defined without much difficulty.  
  The proof can be completed by suitably adjusting the proof of Theorem \ref{thmCapparelli1GenFunc}.  
\end{proof}

One can also alter the conditions of $D_{2,2}(n, m)$ or $D_{2,1}(n, m)$ slightly, 
and construct generating functions for the resulting partitions enumerants.  
It is besides the point whether or not those multiple series are expressible as nice infinite products.  
We give two examples.  

\begin{theorem}
\label{thmDo2-1}
  For $n, m \in \mathbb{N}$, 
  let $D^o_{2, 1}(n, m)$ be the number of partitions of $n$ into $m$ parts 
  in which the smallest part is at least two, 
  the pairwise difference of parts is at least two, 
  and at least three unless the successive parts are both odd.  
  Then, 
  \begin{align*}
%   \label{eqGfDo2-1}
    \sum_{m, n \geq 0} D^o_{2,1}(n, m) q^n x^m 
    = \sum_{n_1, n_2 \geq 0} \frac{ q^{4n_2^2 + 4n_2 + (3n_1^2 + 3n_1)/2 + 4n_2 n_1} (1 + xq^{4n_2 + 2n_1 + 2}) x^{2n_2 + n_1} }
      { (q; q)_{n_1} (q^4; q^4)_{n_2} }
  \end{align*}
\end{theorem}

\begin{proof}
 The idea of the proof is a combination of the proofs of Theorems \ref{thmCapparelli1GenFunc} and \ref{thmGG1-2GenFunc}.  
 We have two types of base partitions as given below.  
 \begin{align}
 \label{baseptnDo2-1beta1}
  \beta = [3, 5], [7, 9], \ldots [4n_2 - 1, 4n_2 + 1], 4n_2 + 3, 4n_2 + 6, \ldots , 4n_2 + 3n_1
 \end{align}
 with weight $4n_2^2 + 4n_2 + (3n_1^2 + 3n_1)/2 + 4n_2 n_1$, or
 \begin{align}
 \label{baseptnDo2-1beta2}
  \beta = 2, [5, 7], [9, 11], \ldots [4n_2 + 1, 4n_2 + 3], 4n_2 + 5, 4n_2 + 8, \ldots , 4n_2 + 3n_1 + 2 
 \end{align}
 with weight $4n_2^2 + 8n_2 + (3n_1^2 + 7n_1)/2 + 4n_2 n_1 + 2$.  
 The base partition \eqref{baseptnDo2-1beta2} gives rise to the desired partitions which contain the part 2, 
 and \eqref{baseptnDo2-1beta1} to which do not.  
 The part 2 in \eqref{baseptnDo2-1beta2} is not moved throughout.  
\end{proof}

\begin{theorem}
\label{thmDe2-2}
  For $n, m \in \mathbb{N}$, 
  let $D^e_{2, 2}(n, m)$ be the number of partitions of $n$ into $m$ parts 
  in which the pairwise difference of parts is at least two, 
  and at least three unless the successive parts are both even.  
  Then, 
  \begin{align*}
%   \label{eqGfDe2-2}
    \sum_{m, n \geq 0} D^e_{2,2}(n, m) q^n x^m 
    = \sum_{n_1, n_2 \geq 0} \frac{ q^{4n_2^2 + 2n_2 + (3n_1^2 + n_1)/2 + 4n_2 n_1} (1 + xq^{4n_2 + 2n_1 + 1}) x^{2n_2 + n_1} }
      { (q; q)_{n_1} (q^4; q^4)_{n_2} }
  \end{align*}
\end{theorem}

\begin{proof}
 This is a corollary of Theorem \ref{thmDo2-1} once we notice that $D^e_{2, 2}(n+m, m) = D^o_{2, 1}(n, m)$.  
\end{proof}

\section{New $q$-series Identities}
\label{secCorollaries} 

Using the substitution $x = 1$ as discussed after \eqref{eqAndrewsGordonSeries}, 
we have the following formulas.  

\begin{cor}
 The following identities hold.  
 \begin{align}
 \label{eqCoroCapparelli1} 
  ( -q^2, -q^3, -q^4, -q^6; q^6)_\infty 
  & = \sum_{n_1, n_2 \geq 0} \frac{ q^{ 2 n_1^2 + 6 n_1 n_2 + 6 n_2^2 } }
      { (q; q)_{n_1} (q^3; q^3)_{n_2} } \\
 \label{eqCoroCapparelli2}
  ( -q, -q^3, -q^5, -q^6; q^6)_\infty 
  & = \sum_{n_1, n_2 \geq 0} \frac{ q^{6n_2^2 + 3n_2 + 2n_1^2 + n_1 + 6n_2 n_1} (1 + q^{3n_2 + 2n_1 + 1}) }
      { (q; q)_{n_1} (q^3; q^3)_{n_2} } \\ 
 \label{eqCoroGG1}
  \frac{1}{( q, q^4, q^7; q^8)_\infty} 
  & =  \sum_{n_1, n_2 \geq 0} \frac{ q^{4n_2^2 + (3n_1^2 - n_1)/2 + 4n_2 n_1} }
      { (q; q)_{n_1} (q^4; q^4)_{n_2} } \\ 
 \label{eqCoroGG2}
  \frac{1}{( q^3, q^4, q^5; q^8)_\infty} 
  & =  \sum_{n_1, n_2 \geq 0} \frac{ q^{4n_2^2 + 4n_2 + (3n_1^2 + 3n_1)/2 + 4n_2 n_1} }
      { (q; q)_{n_1} (q^4; q^4)_{n_2} } 
 \end{align}
\end{cor}

\begin{proof}
 \eqref{eqCoroCapparelli1}, \eqref{eqCoroCapparelli2}, \eqref{eqCoroGG1}, \eqref{eqCoroGG2} 
 are combinations of 
 Theorem \ref{thmCapparelli1} and \eqref{eqCapparelli1GenFunc}, 
 Theorem \ref{thmCapparelli2} and \eqref{eqGFCp2}, 
 Theorem \ref{thmGG1} and \eqref{eqGfGG1}, 
 Theorem \ref{thmGG2} and \eqref{eqGfGG2}, respectively.  
\end{proof}

Above, 
\begin{align*}
 (a_1, a_2, \ldots, a_k; q)_\infty := (a_1; q)_\infty (a_2; q)_\infty \cdots (a_k; q)_\infty.  
\end{align*}

\section{Further Work}
\label{secConclusion}

Kanade and Russell constructed the formulas \eqref{eqCapparelli1GenFunc} and \eqref{eqGFCp2} as well.  
Finite versions of these formulas are presented by Berkovich and Uncu \cite{Berkovich-Uncu}, 
in context of an alternative proof as well as finite versions of Capparelli's identities.  

Kanade-Russell's first four conjectures \cite{KR-conj} can be considered in this context.  
However, for the remaining two, one had better use Gordon marking and related machinery \cite{K-GordonMarking}, 
otherwise the proofs become much longer.  
We leave this to another paper.  

Using the techniques here, 
it was not possible to treat Schur's partition identity \cite{Schur-Schur}.  
The best one could get is either of the series below.  
Here, $s(n, m)$ is the number of partitions of $n$ into $m$ parts 
with pairwise difference at least three, 
and no consecutive multiples of three are allowed.  
\begin{align*}
 \sum_{n, m \geq 0} s(n, m) x^m q^n 
 & = \sum_{n_1, n_2 \geq 0} 
  \frac{ q^{6n_2^2 - n_2 + 2n_1^2 - n_1 + 6n_1n_2} a_{n_2}(q) x^{2n_2 + n_1} }
    { (q; q)_{n_1} } \\
 & = \sum_{n_1, n_2 \geq 0} 
  \frac{ q^{6n_2^2 - n_2 + 2n_1^2 - n_1 + 6n_1n_2} \alpha_{n_2}(q) x^{2n_2 + n_1} }
    { (q; q)_{n_1} (q^3; q^3)_{n_2} }, 
\end{align*}
\begin{align*}
 (1 - q^{3n}) a_n(q) & = (1 + q) a_{n-1}(q) - q a_{n-2}(q) \\ 
 \alpha_n(q) & = (1 + q) \alpha_{n-1}(q) - q (1 - q^{3n-3}) \alpha_{n_2}(q), 
\end{align*}
with the initial conditions $a_0 = \alpha_0 = 1$, 
and $a_n = \alpha_n = 0$ for $n < 0$.  
It is easy to show that $a_n(q)$ has positive coefficients, 
but the same for $\alpha_n(q)$ is not immediately clear.  
The difficulty is that once the base partition 
\begin{align*}
 \beta = [1, 4], [7, 10], \ldots, [6n_2-5, 6n_2-2], 
  6n_2 + 1, 6n_2 + 5, \ldots, 6n_2 + 4n_1 - 3 
\end{align*}
is established, the forward and backward moves of the pairs 
require uneven steps of alternating sizes 3 and 6, 
although they are easily defined.  
To resolve the issue, we need the refined versions of identities in \cite{Boulet}.  
In particular, we need the track of number of parts as well.  

A big collection of partition identities awaiting Andrews-Gordon type $q$-series identities 
may be found in \cite{Lovejoy}.  
The developed machinery does not seem to readily apply, though.  

Another direction may be constructing an evidently positive series for the partitions 
enumerated in Siladi\'{c}'s theorem \cite{Siladic}.

\noindent
{\bf Acknowledgements: }
We thank Karl Mahlburg for the useful discussions and suggestion of terminology, 
Matthew Russell for pointing out \cite{Dousse-Lovejoy}, 
and Ali K. Uncu for the informative workshop talks on \cite{Boulet}.

\bibliographystyle{amsplain}

\end{document}